\documentclass[]{interact}

\usepackage{epstopdf}
\usepackage[caption=false]{subfig}

\usepackage[numbers,sort&compress]{natbib}
\bibpunct[, ]{[}{]}{,}{n}{,}{,}

\theoremstyle{plain}
\newtheorem{theorem}{Theorem}[section]

\theoremstyle{definition}
\newtheorem{definition}[theorem]{Definition}
\newtheorem{example}[theorem]{Example}

\theoremstyle{remark}

\begin{document}

\articletype{ARTICLE TEMPLATE}

\title{On two-term hypergeometric recursions with free lower parameters}

\author{
\name{J.~M.\ Campbell\textsuperscript{a}\thanks{CONTACT J.~M. Campbell. Email: jmaxwellcampbell@gmail.com} 
 and Paul Levrie\textsuperscript{b}}
\affil{\textsuperscript{a} Department of Mathematics, Toronto Metropolitan University, 350 Victoria St., Toronto, Ontario, Canada; 
 \textsuperscript{b} Department of Computer Science, KU Leuven, Celestijnenlaan 200A, 3001 Heverlee (Leuven), Belgium, and Faculty of Applied Engineering, UAntwerpen, Groenenborgerlaan 171, 2020 Antwerp, Belgium}
}

\maketitle

\begin{abstract}
 Let $F(n,k)$ be a hypergeometric function that may be expressed so that $n$ appears within initial arguments of inverted Pochhammer symbols, as in factors of the form $\frac{1}{(n)_{k}}$. Only in exceptional cases is $F(n, k)$ such that Zeilberger's algorithm produces a two-term recursion for $\sum_{k = 0}^{\infty} F(n, k)$ obtained via the telescoping of the right-hand side of a difference equation of the form $p_{1}(n) F(n + r, k) + p_{2}(n) F(n, k) = G(n, k+1) - G(n, k)$ for fixed $r \in \mathbb{N}$ and polynomials $p_{1}$ and $p_{2}$. Building on the work of Wilf, we apply a series acceleration technique based on two-term hypergeometric recursions derived via Zeilberger's algorithm. Fast converging series previously given by Ramanujan, Guillera, Chu and Zhang, Chu, Lupa{\c{s}}, and Amdeberhan are special cases of hypergeometric transforms introduced in our article.
\end{abstract}


\section{Introduction}
Zeilberger's algorithm \cite[\S6]{PetkovsekWilfZeilberger1996} is of great importance in terms of the use of telescoping in the study of and application of 
hypergeometric functions. Recent applications of Zeilberger's algorithm have concerned computational problems in the reduction of radicals 
\cite{Girstmair2021}, the study of families of orthogonal polynomials that generalize the Askey scheme \cite{LamiriWeslati2021}, and the covering of graphs 
with affine hyperplanes \cite{SziklaiWeiner2023}. The importance of Zeilberger's algorithm inside and outside of the field of symbolic computation is evidenced 
by the influence of Paule and Schorn's Mathematica implementation of Zeilberger's algorithm \cite{PauleSchorn1995}, the influence of Koornwinder's work 
concerning the Maple implementation of Zeilberger's algorithm and the $q$-analogue of Zeilberger's algorithm \cite{Koornwinder1993}, and the applications 
of variants and extensions of Zeilberger's algorithm, as in with the extended version of Zeilberger's algorithm due to Chen, Hou, and Mu 
\cite{ChenHouMu2012}. The uses of Zeilberger's algorithm in the symbolic computation discipline motivate the development of Zeilberger-based techniques 
for determining closed forms that are of interdisciplinary significance in terms of how such closed forms may be applied within numerical analysis, and in 
terms of the evaluation of number-theoretic constants. The main purpose of this article is to apply a method based on Zeilberger's algorithm, building on 
the work of Wilf \cite{Wilf1999}, Mohammed \cite{Mohammed2005}, and Hessami Pilehrood and Hessami Pilehrood 
\cite{HessamiPilehroodHessamiPilehrood2008}, for determining recursions for infinite families of infinite hypergeometric series that may be applied so as 
to obtain accelerated series for universal constants such as $\pi$ and $\frac{1}{\pi}$. 

Our recursive approach toward the acceleration of hypergeometric series, as described in Section \ref{sectionmethod} below, also builds on the 
previous work by Chu and Zhang in \cite{ChuZhang2014}, Levrie and Nimbran in \cite{LevrieNimbran2018}, and Campbell and Levrie in 
 \cite{CampbellLevrie2023,LevrieCampbell2023}, in which accelerated hypergeometric series were also obtained via hypergeometric recursions. 
A main advantage as to how our accelerations build on past references as in 
\cite{ChuZhang2014,HessamiPilehroodHessamiPilehrood2008,LevrieCampbell2023,LevrieNimbran2018,Mohammed2005,Wilf1999} is given by 
how broadly our Zeilberger-based technique may be applied to hypergeometric functions involving multiple free parameters 
apart from $n$ and $k$, such as 
\begin{equation}\label{Fakbknknk}
F(n, k) = \frac{(a)_k (b)_k}{ (n)_k^2} 
\end{equation}
and 
\begin{equation}\label{Fakbknk2nk}
F(n, k) = \frac{(a)_k (b)_k}{(n)_k (2 n)_k}
\end{equation}
and the many variants of \eqref{Fakbknknk} and \eqref{Fakbknk2nk}
given in Section \ref{sectionMain} below; 
in this regard, by systematically searching through combinations of rational values for the entries in the tuple $(a, b, n)$ 
such that the hypergeometric sum $\sum_{k=0}^{\infty} F(n, k)$ reduces to a universal constant such as $\pi^{\pm 1}$, 
this is indicative of 
an advantage of our using $F$-functions with multiple free parameters as in \eqref{Fakbknknk} 
and \eqref{Fakbknk2nk}, 
since extending or mimicking the methods of Wilf in \cite{Wilf1999} 
using $F$-functions with multiple free parameters 
allows for a greater range of possibilities for closed-form values of $\sum_{k=0}^{\infty} F(n, k)$, 
 with reference to the constants and convergence rates shown in Table \ref{table:1}, 
together with the motivating examples
under consideration in Section \ref{subsectionMotivating} below. 
 Our techniques and applications inspired by Wilf's 
 acceleration method from \cite{Wilf1999} have not been considered in past research influenced by \cite{Wilf1999}, including 
 \cite{ApagoduZeilberger2006,Chu2021,ChuZhang2014,KondratievaSadov2005,Lupas2000,Mingarelli2013}. 

\begin{table}[h!]
	\centering
	\begin{tabular}{|c c c c|} 
		\hline
		$F(n, k)$ & Convergence rate & Constant(s) & Section \\ [0.5ex] 
		$\frac{(a)_k (b)_k}{ (n)_k^2}$ & $\frac{1}{4}$ & $ \frac{1}{\pi}$, $\frac{\sqrt{2}}{\pi}$, $\frac{\sqrt{3}}{\pi}$, $\ln 2$, 
		$G$, $\sqrt{2}$, $\sqrt{3}$, 
		$\pi^2$ & \ref{subsectionakbknknk} \\ [1.5ex] 
		$\frac{(a)_k (b)_k}{(n)_k (2 n)_k}$ & $ \frac{4}{27} $ & $\frac{1}{\pi}$, 
		$\frac{\sqrt{2}}{\pi}$, $\frac{\sqrt{3}}{\pi}$, $\sqrt[3]{2}$, $\sqrt{2}$, $\sqrt{3}$, $\frac{\pi}{\sqrt{3}}$ & \ref{subsectionakbknk2nk} \\ [1.5ex] 
		$\frac{(a)_k (b)_k}{(n)_k (3 n)_k}$ & $ \frac{27}{256} $ & $\frac{1}{\pi}$, $\frac{\sqrt{2}}{\pi}$ 
		& \ref{subsection256} \\ [1.5ex] 
		$\frac{(a)_k (b)_k}{(2 n)_k (3 n)_k}$ & $ \frac{108}{3125} $ & $\frac{\sqrt{2}}{\pi}$, $\frac{\sqrt{3}}{\pi}$, $\pi$ 
		& \ref{subsection3125} \\ [1.5ex] 
		$\frac{(a)_k (b)_k}{(n)_k (a+n)_k}$ & $ \frac{1}{4} $ & $\frac{\pi}{\sqrt{2}}$, $\frac{\pi}{\sqrt{3}}$, $\sqrt[3]{2}$, $\pi$
		& \ref{subsectionaplusn} \\ [1.5ex] 
		$\frac{(a)_k (b)_k}{(n)_k (a+2 n)_k}$ & $ \frac{4}{27} $ & $\frac{\pi}{\sqrt{2}}$, $\frac{\pi}{\sqrt{3}}$, $\pi$ 
		& \ref{subsectionnka2nk} \\ [1.5ex] 
		$ \frac{(a)_k (b)_k}{(2 n)_k (a+n)_k} $ & $ \frac{4}{27} $ & $\sqrt[3]{2}$, $\sqrt{2}$, $\frac{\pi}{\sqrt{3}}$ 
		& \ref{subsectionakbk2nk} \\ [1.5ex] 
		$ \frac{(-1)^k (a)_k (b)_k}{(a+n+1)_k (b+n+1)_k} 
		$ & $ -\frac{1}{4} $ & $G$, $\sqrt[3]{2}$, $\sqrt{2}$, $\sqrt{3}$ 
		& \ref{subsectionminusquarter} \\ [1.5ex] 
		$ \frac{(a)_{k}^{3}}{(n)_{k}^{3}} 
		$ & $ -\frac{1}{27} $ & $\zeta(3)$ 
		& \ref{subnegative27} \\ [1.5ex] 
		\hline 
	\end{tabular}
	\caption{Organization of Section \ref{sectionMain}, based on the acceleration method described in 
		Section \ref{sectionmethod}. We are letting Catalan's constant
		be denoted as $G = 1 - \frac{1}{3^2} + \frac{1}{5^2} - \cdots$, 
		and we let Ap\'{e}ry's constant be denoted as $\zeta(3) = 1 + \frac{1}{2^3} + \frac{1}{3^3} + \cdots$.}
	\label{table:1}
\end{table}

\section{Motivating results and background}\label{subsectionMotivating}
Our method is such that fast converging series evaluations given by Ramanujan \cite{Ramanujan1914}, Guillera \cite{Guillera2008}, Chu and Zhang 
\cite{ChuZhang2014}, Chu \cite{Chu2011}, Chu \cite{Chu2021}, Lupa{\c{s}} \cite{Lupas2000}, and Amdeberhan \cite{Amdeberhan1996} are special cases of our method. 
In addition to our new, recursive proofs of these past results, we highlight a number of our new results, as below. To begin with, we briefly review some 
necessary preliminaries. 

For $\Re(x) > 0$, the $\Gamma$-function is such that $\Gamma(x)=\int_{0}^{\infty} u^{x-1}e^{-u}\,du$ \cite[\S8]{Rainville1960}. 
Our applications of our techniques to obtain the symbolic evaluations in Section \ref{sectionMain} 
often require the use of the reflection formula 
$$ \Gamma(x) \Gamma(1-x) = \frac{\pi}{\sin(\pi x)} $$ 
and the Gauss multiplication formula 
$$ \Gamma(mx) = (2 \pi)^{\frac{1-m}{2}} m^{mx - \frac{1}{2}} \prod_{k=0}^{m-1} 
\Gamma\left( x + \frac{k}{m} \right). $$ 
The Pochhammer symbol $(x)_{n}$ or rising factorial function is such that $(x)_{n} = \frac{\Gamma(x + n)}{\Gamma(x)}$, 
with $(x)_{n} = x (x + 1) \cdots (x + n - 1)$ for $n \in \mathbb{N}$, and it is not uncommon to write 
\begin{equation*}
\left[ \begin{matrix} \alpha, \beta, \ldots, \gamma \vspace{1mm} \\ 
A, B, \ldots, C \end{matrix} \right]_{n} = \frac{ (\alpha)_{n} (\beta)_{n} 
	\cdots (\gamma)_{n} }{ (A)_{n} (B)_{n} \cdots (C)_{n}}. 
\end{equation*}
As indicated above, our article is devoted to providing and applying new techniques 
concerning recursions for hypergeometric series. In this regard, we recall that generalized hypergeometric series
may be defined so that \cite{Bailey1935} 
\begin{equation}\label{pFqdefinition}
{}_{p}F_{q}\!\!\left[ 
\begin{matrix} 
a_{1}, a_{2}, \ldots, a_{p} \vspace{1mm}\\ 
b_{1}, b_{2}, \ldots, b_{q} 
\end{matrix} \ \Bigg| \ x \right] 
= \sum_{n=0}^{\infty} 
\left[ \begin{matrix} a_{1}, a_{2}, \ldots, a_{p} \vspace{1mm} \\ 
b_{1}, b_{2}, \ldots, b_{q} \end{matrix} \right]_{n} \frac{x^{n}}{n!}, 
\end{equation}
 and we let the argument $x$ in \eqref{pFqdefinition}
 be referred to as the rate of convergence for the series in \eqref{pFqdefinition}. 

We are now in a position to present new formulas introduced in this article as in the following motivating examples: 
\begin{align}
\mbox{(\S {\bf \ref{subsectionakbknk2nk}})\ \ } & \frac{768}{\pi } 
= \sum_{j = 0}^{\infty} \left( \frac{4}{27} \right)^{j} 
\left[ \begin{matrix} 
\frac{1}{2}, \frac{3}{4}, \frac{3}{4}, \frac{5}{4}, \frac{5}{4} \vspace{1mm} \\ 
1, \frac{4}{3}, \frac{5}{3}, 2, 2 
\end{matrix} \right]_{j} (368 j^3+952 j^2+810 j+225), \label{motivatingsimilarCZ} \\
\mbox{(\S {\bf \ref{subsectionminusquarter}})\ \ } & \frac{1215}{7} \sqrt[3]{2} = 
\sum_{j = 0}^{\infty} \left( -\frac{1}{4} \right)^{j} \left[ \begin{matrix} 
\frac{2}{3}, \frac{11}{12}, \frac{17}{12}, \frac{13}{6} \vspace{1mm} \\ 
1, \frac{3}{2}, \frac{7}{4}, \frac{9}{4}
\end{matrix} \right]_{j} \left(360 j^2+678 j+299\right), \label{motivatingcube1} \\ 
\mbox{(\S {\bf \ref{subsectionaplusn}})\ \ } & \frac{55 \pi }{18} 
= \sum_{j = 0}^{\infty} \left( \frac{1}{4} \right)^{j} 
\left[ \begin{matrix} 
\frac{5}{6}, \frac{5}{6}, 1 \vspace{1mm} \\ 
\frac{7}{6}, \frac{17}{12}, \frac{23}{12} 
\end{matrix} \right]_{j} 
(18 j^2+23 j+6), \label{55pi18} \\ 
\mbox{(\S {\bf \ref{subsectionakbknk2nk}})\ \ } & \frac{2187 \sqrt{3}}{4 \pi } 
= \sum_{j = 0}^{\infty} \left( \frac{4}{27} \right)^{j} 
\left[ \begin{matrix} 
\frac{1}{3}, \frac{2}{3}, \frac{5}{6}, \frac{7}{6} \vspace{1mm} \\ 
1, \frac{3}{2}, 2, 2 
\end{matrix} \right]_{j} (621 j^3+1503 j^2+1164 j+280), \label{motivatingbbbbbbb} \\ 
\mbox{(\S {\bf \ref{subsectionminusquarter}})\ \ } & 5 \sqrt{3} = 
\sum_{j = 0}^{\infty} \left( -\frac{1}{4} \right)^{j} \left[ \begin{matrix} 
\frac{1}{4}, \frac{2}{3}, \frac{3}{4}, \frac{5}{6} \vspace{1mm} \\ 
\frac{11}{12}, 1, \frac{17}{12}, \frac{3}{2} 
\end{matrix} \right]_{j} \left(60 j^2+51 j+10\right), \label{5sqrt3} \\
\mbox{(\S {\bf \ref{subsectionaplusn}})\ \ } & \frac{4}{3} \sqrt[3]{2} = {}_{2}F_{1}\!\!\left[ 
\begin{matrix} 
\frac{1}{2}, \frac{2}{3} \vspace{1mm}\\ 
\frac{1}{6} 
\end{matrix} \ \Bigg| \ \frac{1}{4} \right], \label{4thirdscube} \\ 
\mbox{(\S {\bf \ref{subsectionakbknk2nk}})\ \ } & \frac{480 \pi }{7 \sqrt{3}} 
= \sum_{j = 0}^{\infty} \left( \frac{4}{27} \right)^{j} 
\left[ \begin{matrix} 
\frac{2}{3}, 1, \frac{3}{2}, \frac{13}{6} \vspace{1mm} \\ 
\frac{4}{3}, \frac{11}{6}, \frac{7}{3}, \frac{7}{3} 
\end{matrix} \right]_{j} (69 j^3+247 j^2+284 j+104), \label{motivatingpiroot} \\
\mbox{(\S {\bf \ref{subsection256}})\ \ } & \frac{33554432 \sqrt{2}}{105 \pi } 
= \sum_{j = 0}^{\infty} \left( \frac{27}{256} \right)^{j} 
\left[ \begin{matrix} 
\frac{1}{4}, \frac{3}{4}, \frac{11}{12}, \frac{13}{12}, \frac{17}{12}, \frac{19}{12} \vspace{1mm} \nonumber \\ 
1, \frac{3}{2}, \frac{5}{3}, 2, 2, \frac{7}{3} 
\end{matrix} \right]_{j} \cdot \\ 
& \hspace{0.5in} (175872 j^4+684544 j^3+979360 j^2+609312 j+138567), \label{motivatingnewrate} \\
\mbox{(\S {\bf \ref{subsectionakbk2nk}})\ \ } & 45 \sqrt{2} 
= \sum_{j = 0}^{\infty} \left( \frac{4}{27} \right)^{j} 
\left[ \begin{matrix} 
\frac{3}{8}, \frac{3}{8}, \frac{1}{2}, \frac{7}{8}, \frac{7}{8} \vspace{1mm} \\ 
\frac{3}{4}, 1, \frac{13}{12}, \frac{17}{12}, \frac{7}{4} 
\end{matrix} \right]_{j} 
(1472 j^3+1840 j^2+618 j+45), \label{45sqrt2} \\ 
\mbox{(\S {\bf \ref{subnegative27}})\ \ } & 567 \zeta (3) 
= \sum_{j = 0}^{\infty} \left( -\frac{1}{27} \right)^{j} 
\left[ \begin{matrix} 
\frac{1}{2}, \frac{1}{2}, \frac{1}{2}, 1, 1, 1, 1, 1 \vspace{1mm} \\ 
\frac{5}{4}, \frac{5}{4}, \frac{5}{4}, \frac{4}{3}, \frac{5}{3}, \frac{7}{4}, \frac{7}{4}, \frac{7}{4} 
\end{matrix} \right]_{j} \cdot \nonumber \\ 
& \hspace{0.5in} (7168 j^5+23168 j^4+29584 j^3+18620 j^2+5761 j+698). \label{motivatingzeta1}
\end{align}
The above series of convergence rates $\pm \frac{1}{4}$ and $\frac{4}{27}$ have not previously appeared in relevant literature on accelerated series of 
such convergence rates, as in \cite{Chu2011,Chu2021,Chu2021Ramanujan,ChuZhang2014}. Many of the above results are heavily inspired by the accelerated 
series from references due to Chu et al.\ as in \cite{Chu2011,Chu2021,Chu2021Ramanujan,ChuZhang2014}, in which Zeilberger's algorithm is not involved and WZ 
theory is not involved. For example, the motivating example in 
\eqref{motivatingsimilarCZ} recalls Example 106 from \cite{ChuZhang2014}: 
$$ \frac{32}{\pi } = \sum_{j = 0}^{\infty} \left( \frac{4}{27} \right)^{j} 
\left[ \begin{matrix} \frac{1}{2}, \frac{1}{4}, \frac{1}{4}, \frac{3}{4}, \frac{3}{4} \vspace{1mm} \\ 1, 1, 1, \frac{4}{3}, \frac{5}{3} 
\end{matrix} \right]_{j} (368 j^3 + 400 j^2 + 118 j + 9), $$ 
but our acceleration method 
allows us to evaluate series of convergence rates not involved in references as in 
\cite{Chu2011,Chu2021,Chu2021Ramanujan,ChuZhang2014}, 
 and our acceleration method allows us to determine new proofs of and generalizations of results 
 from references as in \cite{Chu2011,Chu2021,Chu2021Ramanujan,ChuZhang2014}. 
 The 2022 version of the Maple Computer Algebra System (CAS) cannot evaluate 
 any of the new series given in this article or the partial sums for such series, 
 as is the case with respect to the 2022 version of the Mathematica CAS. 

In addition to the work of Chu et al.\ as in \cite{Chu2011,Chu2021,Chu2021Ramanujan,ChuZhang2014}, our above formulas for $\frac{1}{\pi}$ and algebraic 
multiples of $\frac{1}{\pi}$ are directly inspired by the famous formulas for expressions of such forms due to Ramanujan 
\cite{Ramanujan1914}. As in each of the articles by Chu et al.\ among 
\cite{Chu2011,Chu2021,Chu2021Ramanujan,ChuZhang2014}, each of which has heavily motivated our explorations based on 
new results as in \eqref{motivatingsimilarCZ}--\eqref{motivatingzeta1}, 
 we highlight the following formulas due to Ramanujan \cite{Ramanujan1914}
as being sources of inspiration underlying our research. Observe that the convergence rates of $\pm \frac{1}{4}$
in the below formulas due to Ramanujan agree with \eqref{motivatingcube1}, \eqref{55pi18}, 
\eqref{5sqrt3}, and \eqref{4thirdscube}: 
\begin{align}
& \frac{4}{\pi} = \sum_{n=0}^{\infty} \left( \frac{1}{4} \right)^{n} \left[ \begin{matrix} 
\frac{1}{2}, \frac{1}{2}, \frac{1}{2} \vspace{1mm} \\ 
1, 1, 1 
\end{matrix} \right]_{n} \left(6n + 1\right), \label{Ramanujan1} \\ 
 & \frac{8}{\pi} = 
\sum_{n=0}^{\infty} \left( -\frac{1}{4} \right)^{n} \left[ \begin{matrix} 
\frac{1}{4}, \frac{1}{2}, \frac{3}{4} \vspace{1mm} \\ 
1, 1, 1 
\end{matrix} \right]_{n} \left(20n + 3\right). \nonumber 
\end{align}
Our acceleration method described in Section \ref{sectionmethod} and inspired by Wilf's work in \cite{Wilf1999} 
is such that we may apply this method to formulate a new proof of Ramanujan's formula in \eqref{Ramanujan1}. 
This recalls Zeilberger's proof via the WZ method, as opposed to Zeilberger's algorithm, 
for Ramanujan's ${}_{4}F_{3}(-1)$-series for $\frac{1}{\pi}$ \cite{EkhadZeilberger1994}, 
along with Guillera's WZ proof of \eqref{Ramanujan1} \cite{Guillera2006}. 

The ${}_{2}F_{1}\left( \frac{1}{4} \right)$-evaluation shown in \eqref{4thirdscube} is especially notable, in terms of how it stands out against the 
neighbouring formulas among \eqref{motivatingsimilarCZ}--\eqref{5sqrt3} 
and \eqref{motivatingpiroot}, 
since a (nontrivial) polynomial in $n$ is not involved in 
the summand of \eqref{4thirdscube}. 
What is even more striking about the formula for $\sqrt[3]{2}$ in 
\eqref{4thirdscube} is due to how it relates
to the below listed ${}_{2}F_{1}$-formulas from Zucker and Joyce \cite{JoyceZucker2002,ZuckerJoyce2001} that are highlighted 
in the Wolfram MathWorld entry on hypergeometric functions. In this encyclopedia entry, 
 it is suggested how curious it is how ${}_{2}F_{1}$-functions can assume integer root values
for very specific rational arguments and rational parameters, as below, 
again noting the resemblance to our new formula in \eqref{4thirdscube}: 
\begin{align*}
& \frac{2}{3} \sqrt{7} = {}_{2}F_{1}\!\!\left[ 
\begin{matrix} 
\frac{1}{8}, \frac{3}{8} \vspace{1mm}\\ 
\frac{1}{2} 
\end{matrix} \ \Bigg| \ \frac{2400}{2401} \right], \\
& \frac{3}{4} \sqrt{3} = {}_{2}F_{1}\!\!\left[ 
\begin{matrix} 
\frac{1}{6}, \frac{1}{3} \vspace{1mm}\\ 
\frac{1}{2} 
\end{matrix} \ \Bigg| \ \frac{25}{27} \right], \\ 
& \frac{4}{3} 2^{1/6} = {}_{2}F_{1}\!\!\left[ 
\begin{matrix} 
\frac{1}{6}, \frac{1}{2} \vspace{1mm}\\ 
\frac{2}{3} 
\end{matrix} \ \Bigg| \ \frac{125}{128} \right], \\ 
& \frac{3}{4} 11^{1/4} = {}_{2}F_{1}\!\!\left[ 
\begin{matrix} 
\frac{1}{12}, \frac{5}{12} \vspace{1mm}\\ 
\frac{1}{2} 
\end{matrix} \ \Bigg| \ \frac{1323}{1331} \right]. 
\end{align*}

 The above series for $\frac{\sqrt{2}}{\pi}$ of convergence rate $\frac{27}{256}$ shown in \eqref{motivatingnewrate} is 
 also of especial interest. 
 To the best of our knowledge, the only series of convergence rate $\frac{27}{256}$ 
 that have previously appeared are the following series given by Chu \cite{Chu2011} 
 via Dougall's bilateral ${}_{2}H_{2}$-sum. 
 These past series motivate the new Zeilberger-based techniques we have introduced for generating and proving 
 evaluations for series of convergence rate $\frac{27}{256}$: 
\begin{align*}
& \frac{1287 \sqrt{3}}{\pi } 
= \sum_{j = 0}^{\infty} \left( \frac{27}{256} \right)^{j} 
\left[ \begin{matrix} 
\frac{1}{9}, \frac{2}{9}, \frac{4}{9}, \frac{5}{9}, \frac{7}{9}, \frac{8}{9} \vspace{1mm} \\ 
1, 1, 1, \frac{3}{2}, \frac{5}{4}, \frac{7}{4} 
\end{matrix} \right]_{j} \cdot \\ 
& \hspace{1.5in} \text{\footnotesize $(166941j^4 + 260253 j^3 + 130464 j^2 + 23202 j + 1120)$}, \\
& 105 \pi 
= \sum_{j = 1}^{\infty} \left( \frac{27}{256} \right)^{j} 
\left[ \begin{matrix} 
1, \frac{1}{2}, \frac{1}{6}, \frac{1}{6}, \frac{5}{6}, \frac{5}{6} \vspace{1mm} \\ 
\frac{4}{3}, \frac{5}{3}, \frac{9}{8}, \frac{11}{8}, \frac{13}{8}, \frac{15}{8} 
\end{matrix} \right]_{j} \cdot \\ 
& \hspace{1.5in} \text{\footnotesize $(16488j^5 + 42192 j^4 + 40606 j^3 + 18247j^2 + 3842 j + 317)/j$}, \\
& 60 \pi - 149 
= \sum_{j = 1}^{\infty} \left( \frac{27}{256} \right)^{j} 
\left[ \begin{matrix} 
1, \frac{1}{2}, \frac{1}{6}, \frac{1}{6}, \frac{5}{6}, \frac{5}{6} \vspace{1mm} \\ 
\frac{4}{3}, \frac{5}{3}, \frac{7}{8}, \frac{9}{8}, \frac{11}{8}, \frac{13}{8} 
\end{matrix} \right]_{j} \cdot \\ 
& \hspace{1.5in} \text{\footnotesize $(32976j^5 + 67896j^4 + 52844j^3 + 19156j^2 + 3167 j +180)/j$}. 
\end{align*}

We have adopted a ``Chu-style'' \cite{CampbellMaple}
way of denoting hypergeometric series in the vein of \cite{Chu2011,Chu2021,Chu2021Ramanujan,ChuZhang2014}, 
by expressing universal constants in the form 
\begin{equation}\label{Chustyle}
\sum_{n=0}^{\infty} x^n \left[ \begin{matrix} 
\alpha, \beta, \ldots, \gamma \vspace{1mm} \\ 
A, B, \ldots, C 
\end{matrix} \right]_{n} p(n), 
\end{equation}
for a fixed argument $x$ and for a polynomial (or rational function) 
$p(n)$, and where for a lower Pochhammer symbol 
$ (\ell)_{n}$ in \eqref{Chustyle}, it is not the case that there is an upper Pochhammer symbol $(u)_{n}$ in \eqref{Chustyle} such that $\ell - u$ is 
a member of $\mathbb{Z}$. 

 Our interest
in the series for $\frac{\sqrt{2}}{\pi}$ shown in \eqref{motivatingnewrate}
is also motivated by Ramanujan's famous series for the same value of $\frac{\sqrt{2}}{\pi}$, as below \cite{Ramanujan1914}: 
\begin{align*}
& \frac{\sqrt{2}}{\pi} 
= \frac{4}{9} \sum_{n=0}^{\infty} \left( \frac{1}{3} \right)^{4 n} 
\left[ \begin{matrix} 
\frac{1}{2}, \frac{1}{4}, \frac{3}{4} \vspace{1mm} \\ 
1, 1, 1 
\end{matrix} \right]_{n} (10 n + 1), \\ 
& \frac{\sqrt{2}}{\pi} 
= \frac{4}{9801} \sum_{n=0}^{\infty} \left( \frac{1}{99} \right)^{4 n} 
\left[ \begin{matrix} 
\frac{1}{2}, \frac{1}{4}, \frac{3}{4} \vspace{1mm} \\ 
1, 1, 1 
\end{matrix} \right]_{n} (26390 n + 1103). 
\end{align*}

\section{A Zeilberger-based series acceleration method}\label{sectionmethod}
For the sake of brevity, we assume some basic familiarity with 
Zeilberger's algorithm \cite[\S6]{PetkovsekWilfZeilberger1996} and with bivariate hypergeometric functions and the like. 
We have systematically performed computational experiments with Maple using the package given by the input 
\begin{verbatim}
with(SumTools[Hypergeometric]):
\end{verbatim}
and with the use of bivariate hypergeometric functions $F(n, k)$ that satisfy the following: 

\begin{enumerate}
	
	\item The expression $F(n, k)$ satisfies a first-order recurrence according to Zeilberger's algorithm, with respect to $n$; 
	
	\item The expression $F(n, k)$ may be written so as to contain a factor given by the reciprocal 
	of a Pochhammer symbol involving $n$ as its initial argument; 
	
	\item Letting the phrase \emph{$n$-Pochhammer expression} refer to a Pochhammer symbol $(\alpha)_{\beta}$ such that $\alpha$ is a rational linear 
	combination of variables that include $n$ and such that the coefficient of $n$ is nonzero, the number of lower $n$-Pochhammer expressions in $F(n, k)$ 
	minus the number of upper $n$-Pochhammer expressions in $F(n, k)$ is at least $2$; 
	
	\item Letting $G(n, k)$ denote the companion to $F(n, k)$ obtained via Zeilberger's algorithm, the limit $\lim_{\ell \to \infty} G(n, \ell) $ vanishes for sufficiently 
	large $n$; and 
	
	\item The expression $G(n, 0)$ is non-vanishing. 
	
\end{enumerate}

Given a function $F(n, k)$ satisfying the above conditions, we may write 
\begin{equation}\label{maindifference}
p_{1}(n) F(n + r, k) + p_{2}(n) F(n, k) = G(n, k+1) - G(n, k) 
\end{equation}
for (nonzero) polynomials $p_{1}$ and $p_{2}$ with integer coefficients and for fixed $r \in \mathbb{N}$ 
(typically with $r = 1$). Since this construction is closely related to what is meant by a \emph{Markov--WZ pair}, 
it is appropriate to review the definition of this term given in references such as 
\cite{HessamiPilehroodHessamiPilehrood2008,Mohammed2005}, 
recalling that a function that satisfies a linear recurrence with polynomial coefficients is said to be \emph{P-recursive}. 

\begin{definition}
	A pair $(F, G)$ of bivariate hypergeometric functions is said to be a \emph{Markov--WZ} pair if there 
	is a polynomial $p(n, k)$ of the form 
	$$ p(n, k) = a_{0}(n) + a_{1}(n) k + \cdots + a_{\ell}(n) k^{\ell} $$ 
	for fixed $\ell \in \mathbb{N}_{0}$ and P-recursive functions $a_{0}(n)$, $a_{1}(n)$, $\ldots$, $a_{\ell}(n)$ whereby 
	\begin{equation}\label{Markovdefinition}
	p(n+1, k) F(n+1,k) - p(n, k) F(n, k) = G(n, k+1) - G(n, k). 
	\end{equation}
\end{definition}

So, for the $r = 1$ case of \eqref{maindifference}, 
a pair $(F, G)$ satisfying \eqref{maindifference}
is closely related to what is meant by a Markov--WZ pair, 
but we do not insist that $p_{1}(n) = p_{2}(n+1)$, 
noting the distinction between the difference equation in \eqref{maindifference} 
and the difference equation in \eqref{Markovdefinition}. 

By writing $f(n) = \sum_{k=0}^{\infty} F(n, k)$, a telescoping phenomenon applied to 
\eqref{maindifference} gives us that 
\begin{equation}\label{telescopemain}
p_{1}(n) f(n + r) + p_{2}(n) f(n) = - G(n, 0), 
\end{equation}
which we rewrite as 
\begin{equation}\label{generalf}
f(n) = \frac{- G(n, 0)}{p_{2}(n)} - \frac{p_{1}(n)}{p_{2}(n)} f(n + r). 
\end{equation}
The repeated application of the $f$-recursion in \eqref{generalf} typically has the effect of accelerating the convergence of $f$ 
(cf.\ \cite{Wilf1999}), as explored in this article, 
again for hypergeometric expressions $F(n, k)$ satisfying the above listed conditions. Wilf, in \cite{Wilf1999}, introduced the idea of accelerating series 
using recursions as in \eqref{generalf} derived via Zeilberger's algorithm, and Mohammed \cite{Mohammed2005} explored related accelerated 
techniques. 
In this article, we pursue a full exploration of the acceleration of series using \eqref{generalf} for functions $F(n, k)$ satisfying the 
above conditions and with free parameters apart from $n$ and $k$ involved. 
This has led us to obtain many new hypergeometric transforms that are not considered in Wilf's 
article \cite{Wilf1999} or in related work on Markov--WZ pairs as in \cite{HessamiPilehroodHessamiPilehrood2008,Mohammed2005}. 

Our experimental use of the Maple CAS has led us to discover how our acceleration technique may be broadly applied to functions $F(n, k)$ of the forms 
indicated as follows: 
\begin{equation}\label{alphabetaalphabeta}
F(n, k) = \frac{ (a)_{k} (b)_{k} }{ (\alpha_{1} + \beta_{1} n)_{k} (\alpha_{2} + \beta_{2} n)_{k} } 
\end{equation}
for free parameters $a$ and $b$ and for fixed parameters $\alpha_{1}$, $\beta_{1}$, $\alpha_{2}$, and $\beta_{2}$ that we typically set so that $ 
\beta_{1}, \beta_{2} \in \mathbb{N}$. Our application of our Zeilberger-based acceleration technique using \eqref{alphabetaalphabeta} is such that 
$r = 1$ for the recursion indicated in \eqref{generalf} and for most of the hypergeometric functions $F(n, k)$ considered 
in this article. In such cases, by rewriting \eqref{generalf} as 
\begin{equation}\label{mainrecurrence}
f(n) = g_{1}(n) + g_{2}(n) f(n+1) 
\end{equation}
for rational functions $g_{1}(n)$ and $g_{2}(n)$, the inductive application of \eqref{mainrecurrence} gives us that 
\begin{equation}\label{mainprod}
f(n) = \sum _{j=0}^\infty \left(\prod _{i=0}^{j-1} g_{2}(n + i)\right) g_{1}(n + j) 
\end{equation}
if 
\begin{equation}\label{requiredvanishing}
\lim_{m \to \infty} \prod_{i=0}^{m} g_{2}(n + i) f(n + m + 1) 
\end{equation}
vanishes (cf.\ \cite{Wilf1999}), in which case the equation in \eqref{mainprod} gives us series accelerations for the hypergeometric functions $F(n, k)$ 
considered in this article. It is appropriate to refer to a recursion as in \eqref{mainrecurrence} as a two-term recursion, and this is consistent with 
combinatorial work on general two-term recurrences as in \cite{MansourShattuck2013}. 

 The organization of Section \ref{sectionMain} is summarized in Table \ref{table:1}. The organization of our article, with 
 reference to Table \ref{table:1} and a similarly organized Table in \cite{LevrieCampbell2023}, is inspired by our past work on hypergeometric recursions 
and series accelerations \cite{LevrieCampbell2023}, which did not involve Zeilberger's algorithm, and which relied on ``ad hoc'' and experimentally 
discovered recursions, compared to how systematically we can apply the method given in our current article. 

\section{Main results}\label{sectionMain}
Each of the following subsections is based on a given selection of a hypergeometric function $F$ satisfying the conditions listed 
in Section \ref{sectionmethod}. 

\subsection{An acceleration based on $\frac{(a)_k (b)_k}{ (n)_k^2}$}\label{subsectionakbknknk}
We begin by setting $F(n, k)$ to be as in \eqref{Fakbknknk}, 
and by determining that the function in \eqref{Fakbknknk} 
satisfies the first out of the conditions listed in 
Section \ref{sectionmethod}. 
In this regard, we may apply the Maple implementation of Zeilberger's algorithm to compute the 
companion function $G(n) = R(n, k) F(n, k)$, where 
\begin{equation}\label{Rakbknknk}
R(n, k) = -n^2 \left(-2 n (a+b-k+1)+a b-a k+a-b k+b+3 n^2\right). 
\end{equation}
This leads us toward the following recursion that we are to apply via 
accelerations of series of the form $\sum_{k=0}^{\infty} F(n, k)$. 

\begin{theorem}\label{firstmaintheorem}
	Letting $f(n)$ denote the hypergeometric series $\sum_{k=0}^{\infty} F(n, k)$, where $F(n, k)$ is as in \eqref{Fakbknknk}, the recursion $$ f(n) = 
	\frac{a+b+a b-2 (1+a+b) n+3 n^2}{(a+b-2 n) (1+a+b-2 n)} + \frac{ (a-n)^2 (b - n)^2 }{(a + b - 2 n) (1+a+b-2 n) n^2} f(n+1) $$
	holds if the above 
	series converge. 
\end{theorem}

\begin{proof}
	Using Zeilberger's algorithm, we may determine that 
	\begin{align*}
	& -(a-n)^2 (b-n)^2 F(n+1,k)+n^2 (-1-a-b+2 n) (-a-b+2 n) F(n,k) \\
	& = G(n,k+1)-G(n,k), 
	\end{align*}
	again writing $G(n, k) = F(n, k) R(n, k)$ for the rational function $R(n, k)$ indicated in \eqref{Rakbknknk}. 
 Applying the partial sum operator $\sum_{k=0}^{m} \cdot$ to both sides of the above equality, 
 the right-hand side telescopes, so that 
	\begin{align*}
	& -(a-n)^2 (b - n)^2 \sum_{k=0}^{m} F(n+1,k) + 
	n^2 (-1-a-b+2 n) (-a-b+2 n) \sum_{k=0}^{m} F(n,k) \\ 
	& = G(n,m+1)-G(n, 0), 
	\end{align*}
	we let $m \to \infty$, with $\lim_{m \to \infty} G(n,m+1)$ vanishing for $n$
	such that $ \sum_{k=0}^{\infty} F(n,k)$ converges. This gives us an equivalent formulation of the desired result, 
	since it is easily seen that \eqref{requiredvanishing} vanishes. 
\end{proof}

Setting $$ g_{1}(n) = \frac{a+b+a b-2 (1+a+b) n+3 n^2}{(a+b-2 n) (1+a+b-2 n)} $$ 
and $$ g_{2}(n) = \frac{(a-n)^2 (b-n)^2}{(a+b-2 n) (1+a+b-2 n) n^2}, $$ 
repeated applications of the hypergeometric recursion in Theorem \ref{firstmaintheorem}
give us that a recursion of the form indicated in \eqref{mainprod} holds. 
This gives us a series acceleration identity that we are to apply as below. 

\begin{example}
	Setting $(a, b, n) = \left( \frac{1}{2}, \frac{1}{2}, 2 \right)$, we obtain a new, Zeilberger-based proof of Ramanujan's formula in \eqref{Ramanujan1}. 
\end{example}

\begin{example}
	Setting $(a, b, n) = \left( \frac{1}{2}, \frac{1}{2}, \frac{3}{2} \right)$, we obtain Guillera's formula 
	\begin{equation}\label{GuillerainChu}
	\frac{\pi^2}{4} = 
	\sum_{n=0}^{\infty} \left( \frac{1}{4} \right)^{n} \left[ \begin{matrix} 
	1, 1, 1 \vspace{1mm} \\ 
	\frac{3}{2}, \frac{3}{2}, \frac{3}{2} 
	\end{matrix} \right]_{j} (3n + 2), 
	\end{equation}
	introduced in \cite{Guillera2008}. 
\end{example}

\begin{example}
	Setting $(a, b, n) = \left( \frac{1}{2}, 1, \frac{3}{2} \right)$, we obtain a new proof of Example 84 
	from Chu and Zhang's article \cite{ChuZhang2014}: 
	\begin{equation*}
	6 G = 
	\sum_{j = 0}^{\infty} \left( \frac{1}{4} \right)^{n} \left[ \begin{matrix} 
	1, 1 \vspace{1mm} \\ 
	\frac{5}{4}, \frac{7}{4} 
	\end{matrix} \right]_{j} \frac{6 j+5}{2 j+1}. 
	\end{equation*}
\end{example}

\begin{example}
	Setting $(a, b, n) = \left( \frac{3}{4}, \frac{5}{4}, 2 \right)$, 
	$$ \frac{2 \sqrt{2}}{\pi } 
	= \sum_{j = 0}^{\infty} \left( \frac{1}{4} \right)^{j} 
	\left[ \begin{matrix} 
	-\frac{1}{4}, -\frac{1}{4}, \frac{1}{4}, \frac{1}{4} \vspace{1mm} \\ 
	\frac{1}{2}, 1, 1, 1
	\end{matrix} \right]_{j} (1-48 j^2). $$
\end{example}

\begin{example}
	Setting $(a, b, n) = \left( \frac{3}{2}, 1, 2 \right)$, we obtain that 
	$$ 3 \ln (2) 
	= \sum_{j = 0}^{\infty} \left( \frac{1}{4} \right)^{j} 
	\left[ \begin{matrix} 
	\frac{1}{2}, \frac{1}{2} \vspace{1mm} \\ 
	\frac{5}{4}, \frac{7}{4}
	\end{matrix} \right]_{j} \frac{3 j+2}{j+1}. $$
\end{example}

\begin{example}
	Setting $(a, b, n) = \left( \frac{1}{3}, \frac{2}{3}, 2 \right)$, we obtain a new proof of 
	Example 9 from Chu's article \cite{Chu2011}, which is reproduced below: 
	$$ \frac{9 \sqrt{3}}{2 \pi } 
	= \sum_{j = 0}^{\infty} \left( \frac{1}{4} \right)^{j} 
	\left[ \begin{matrix} 
	\frac{1}{3}, \frac{1}{3}, \frac{2}{3}, \frac{2}{3} \vspace{1mm} \\ 
	1, 1, 1, \frac{3}{2}
	\end{matrix} \right]_{j} (27 j^2+18 j+2). $$ 
\end{example}

\begin{example}
	Setting $(a, b, n) = \left( \frac{1}{3}, \frac{2}{3}, \frac{3}{2} \right)$, we may obtain that 
	$$ \frac{\sqrt{3}}{2} 
	= \sum_{j = 0}^{\infty} \left( \frac{1}{4} \right)^{j} 
	\left[ \begin{matrix} 
	-\frac{1}{6}, -\frac{1}{6}, \frac{1}{6}, \frac{1}{6} \vspace{1mm} \\ 
	\frac{1}{2}, \frac{1}{2}, \frac{1}{2}, 1
	\end{matrix} \right]_{j} (1+36 j-108 j^2). $$
\end{example}

\begin{example}
	Setting $(a, b, n) = \left( \frac{1}{6}, \frac{5}{6}, 2 \right)$, we may obtain a new proof of Example 15 from Chu's article 
	\cite{Chu2011}: $$ \frac{18}{\pi } = \sum_{j = 0}^{\infty} \left( \frac{1}{4} \right)^{j} \left[ \begin{matrix} \frac{1}{6}, \frac{1}{6}, 
	\frac{5}{6}, \frac{5}{6} \vspace{1mm} \\ 1, 1, 1, \frac{3}{2} \end{matrix} \right]_{j} (108 j^2+72 j+5). $$ 
\end{example}

\begin{example}
	Setting $(a, b, n) = \left( \frac{1}{4}, \frac{3}{4}, \frac{3}{2} \right)$, we may obtain that 
	$$ \frac{\sqrt{2}}{2} 
	= \sum_{j = 0}^{\infty} \left( \frac{1}{4} \right)^{j} 
	\left[ \begin{matrix} 
	-\frac{1}{4}, -\frac{1}{4}, \frac{1}{4}, \frac{1}{4} \vspace{1mm} \\ 
	\frac{1}{2}, \frac{1}{2}, \frac{1}{2}, 1
	\end{matrix} \right]_{j} (1+16 j-48 j^2). $$
\end{example}

By setting 
\begin{equation}\label{2092830322149521PM11A}
F(n, k) = \frac{(a)_k (b)_k}{(k+n) (n)_k^2}, 
\end{equation}
we may obtain many variants of the above series, 
and many similar generalizations by generalizing \eqref{2092830322149521PM11A}. 

\begin{example}
	Setting $(a, b, n) = \left( -\frac{3}{4}, -\frac{1}{4}, 1 \right)$ in the acceleration derived from
	\eqref{2092830322149521PM11A}, we may obtain that 
	$$ \frac{16384 \sqrt{2}}{105 \pi } 
	= \sum_{j = 0}^{\infty} \left( \frac{1}{4} \right)^{j} 
	\left[ \begin{matrix} 
	\frac{5}{4}, \frac{7}{4}, \frac{9}{4}, \frac{11}{4} \vspace{1mm} \\ 
	1, 2, \frac{5}{2}, 3 
	\end{matrix} \right]_{j} (16 j^2+48 j+33). $$ 
\end{example}

\subsection{An acceleration based on $\frac{(a)_k (b)_k}{(n)_k (2 n)_k}$}\label{subsectionakbknk2nk}
We refer to Chu and Zhang's article \cite{ChuZhang2014} on hypergeometric series accelerations derived via the modified Abel lemma and Dougall's 
${}_{5}F_{4}$-sum, together with references therein, such as the seminal work by the Chudnovsky brothers \cite{ChudnovskyChudnovsky1988} and 
the excellent survey by Baruah et al.\ on Ramanujan's series for $\frac{1}{\pi}$ \cite{BaruahBerndtChan2009}, for background material on the historical importance 
of the computational problems concerning the fundamental constant $\pi$. Since we have applied our acceleration method to obtain a new proof of 
Ramanujan's series for $\frac{1}{\pi}$ in \eqref{Ramanujan1}, a new proof of Guillera's series for $\pi^2$ in \eqref{GuillerainChu}, as well as new proofs 
of series for $\frac{1}{\pi}$ and $\frac{\sqrt{3}}{\pi}$ given by Chu and Zhang \cite{ChuZhang2014}, 
with Theorem \ref{firstmaintheorem} and the corresponding acceleration formula in 
 \eqref{mainprod} providing infinite generalizations for all of these past formulas, 
 this inspires further applications of the method from Section \ref{sectionmethod}, 
to obtain faster converging series relative to Section \ref{subsectionakbknknk}. 
In this regard, by setting $F(n, k)$ to be as in \eqref{Fakbknk2nk}, 
and by again applying the acceleration technique indicated in Section \ref{sectionmethod}, 
this leads us to a hypergeometric transform that may be used to generate series of convergence rate $\frac{4}{27} \approx 0.148148$. 

\begin{theorem}
	Letting $f(n) = \sum_{k=0}^{\infty} F(n, k)$ for the hypergeometric function $F(n, k)$ 
	in \eqref{Fakbknk2nk}, 
	the recursion $f(n) = g_{1}(n) + g_{2}(n) f(n+1)$ holds, for 
	$$g_{2}(n) = -\frac{(a-2 n-1) (a-2 n) (a-n) (b-2 n-1) (b-2 n) (b-n)}{2 n^2 (2 n+1) (a+b-3 n-1) (a+b-3 n) (a+b-3 n+1)}$$
	and for $g_{1}(n)$ equals to the expression given by the following Mathematica output. 
	\begin{verbatim}
	-((a n (2 - b (2 + 5 b) + 6 n + 27 b n - 38 n^2) + 
	a^2 (b^2 - 5 b n + 2 n (-1 + 4 n)) + 
	2 n (b + b (3 - 19 n) n + b^2 (-1 + 4 n) + 
	n (-3 + n (-2 + 23 n))))/(
	2 (-1 + a + b - 3 n) (a + b - 3 n) (1 + a + b - 3 n) n))
	\end{verbatim}
\end{theorem}

\begin{proof}
	We again set $F(n, k)$ to be as indicated in \eqref{Fakbknk2nk}. Applying Zeilberger's algorithm, we obtain a companion function $G(n, k) = R(n, k) F(n, 
	k)$, for a rational function $R(n, k)$ given by the following Mathematica output. 
	\begin{verbatim}
	(1/(k + 2 n))2 n^2 (1 + 
	2 n) ((-1 + b) b (-1 + k) k + (3 (-1 + k) k + b^2 (-2 + 5 k) + 
	b (2 + k - 6 k^2)) n + (-6 + 8 b^2 + b (6 - 28 k) + 
	k (4 + 9 k)) n^2 + (-4 - 38 b + 39 k) n^3 + 46 n^4 + 
	a (-k (-1 + b + b^2 + k - 2 b k) + (2 - b (2 + 5 b) + k + 13 b k - 
	6 k^2) n + (6 + 27 b - 28 k) n^2 - 38 n^3) + 
	a^2 (b^2 + (-1 + k) k + (-2 + 5 k) n + 8 n^2 - b (k + 5 n)))
	\end{verbatim}
	Moreover, Zeilberger's algorithm gives us that a difference equation of the form indicated in \eqref{maindifference} holds for $r = 1$. In this case, the 
	polynomial $p_{1}(n)$ is given by the following Mathematica output. 
	\begin{verbatim}
	(a^3*b^3 - 5*a^3*b^2*n + 8*a^3*b*n^2 - 4*a^3*n^3 - 5*a^2*b^3*n + 
	25*a^2*b^2*n^2 - 40*a^2*b*n^3 + 20*a^2*n^4 + 8*a*b^3*n^2 - 
	40*a*b^2*n^3 + 64*a*b*n^4 - 32*a*n^5 - 4*b^3*n^3 + 20*b^2*n^4 - 
	32*b*n^5 + 16*n^6 - a^3*b^2 + 3*a^3*b*n - 2*a^3*n^2 - a^2*b^3 + 
	10*a^2*b^2*n - 23*a^2*b*n^2 + 14*a^2*n^3 + 3*a*b^3*n - 
	23*a*b^2*n^2 + 48*a*b*n^3 - 28*a*n^4 - 2*b^3*n^2 + 14*b^2*n^3 - 
	28*b*n^4 + 16*n^5 + a^2*b^2 - 3*a^2*b*n + 2*a^2*n^2 - 3*a*b^2*n + 
	9*a*b*n^2 - 6*a*n^3 + 2*b^2*n^2 - 6*b*n^3 + 4*n^4)
	\end{verbatim}
	The corresponding polynomial $p_{2}(n)$ is given by the following Mathematica output. 
	\begin{verbatim}
	(-108*n^6 + 4*a^3*n^3 - 36*a^2*n^4 + 108*a*n^5 + 2*a^3*n^2 - 
	18*a^2*n^3 - 54*n^5 + 6*a^2*b*n^2 + 6*a*b^2*n^2 - 36*a*b*n^3 - 
	2*b*n^2 - 4*b*n^3 + 12*a*b^2*n^3 + 12*a^2*b*n^3 - 72*a*b*n^4 - 
	2*a*n^2 + 12*n^4 + 6*n^3 + 4*b^3*n^3 - 36*b^2*n^4 + 108*b*n^5 + 
	2*b^3*n^2 - 18*b^2*n^3 + 54*b*n^4 + 54*a*n^4 - 4*a*n^3)
	\end{verbatim}
	Manipulating the difference equation in \eqref{maindifference}
	so as to obtain \eqref{telescopemain} for $r = 1$, 
	the equivalent expression in \eqref{generalf} is equivalent to the desired result, 
	as it is easily seen that \eqref{requiredvanishing} vanishes. 
\end{proof}

\begin{example}
	Setting $(a, b, n) = \left( \frac{1}{4}, \frac{3}{4}, 1 \right)$, we may obtain that 
	\begin{align*}
	& \frac{8192 \sqrt{2}}{\pi } = \\
	& \sum_{j = 0}^{\infty} \left( \frac{4}{27} \right)^{j} 
	\left[ \begin{matrix} 
	\frac{1}{4}, \frac{5}{8}, \frac{3}{4}, \frac{7}{8}, \frac{9}{8}, \frac{11}{8} \vspace{1mm} \\ 
	1, \frac{4}{3}, \frac{3}{2}, \frac{5}{3}, 2, 2 
	\end{matrix} \right]_{j} (11776 j^4+36352 j^3+40976 j^2+19856 j+3465). 
	\end{align*}
\end{example}

\begin{example}
	Setting $(a, b, n) = \left( \frac{1}{2}, \frac{1}{2}, 1 \right)$, we may obtain the motivating example highlighted in \eqref{motivatingsimilarCZ}. 
\end{example}

\begin{example}
	Setting $(a, b, n) 
	= \left( \frac{1}{3}, \frac{2}{3}, 1 \right)$, we may obtain 
	the motivating example highlighted in \eqref{motivatingbbbbbbb}. 
\end{example}

\begin{example}
	Setting $(a, b, n) 
	= \left( \frac{1}{6}, \frac{5}{6}, 1 \right)$, we may obtain that 
	\begin{align*}
	\frac{279936}{5 \pi } = \sum_{j = 0}^{\infty} \left( \frac{4}{27} \right)^{j} &
	\left[ \begin{matrix} 
	\frac{1}{6}, \frac{7}{12}, \frac{5}{6}, \frac{11}{12}, \frac{13}{12}, \frac{17}{12} \vspace{1mm} \\ 
	1, \frac{4}{3}, \frac{3}{2}, \frac{5}{3}, 2, 2 
	\end{matrix} \right]_{j} \cdot \\
	& (59616 j^4+184032 j^3+206748 j^2+99324 j+17017). 
	\end{align*}
\end{example}

\begin{example}
	Setting $(a, b, n) 
	= \left( -\frac{3}{4}, -\frac{1}{4}, \frac{1}{2} \right)$, we may obtain that 
	\begin{align*}
	3360 \sqrt{2} 
	= \sum_{j = 0}^{\infty} \left( \frac{4}{27} \right)^{j} & 
	\left[ \begin{matrix} 
	\frac{5}{8}, \frac{3}{4}, \frac{7}{8}, \frac{9}{8}, \frac{5}{4}, \frac{11}{8} \vspace{1mm} \\ 
	\frac{1}{2}, 1, \frac{3}{2}, \frac{3}{2}, \frac{11}{6}, \frac{13}{6} 
	\end{matrix} \right]_{j} \cdot \\
	& (11776 j^4+32256 j^3+30224 j^2+11232 j+1365). 
	\end{align*}
\end{example}

\begin{example}
	Setting $(a, b, n) = \left( -\frac{2}{3}, \frac{2}{3}, \frac{4}{3} \right)$, we may obtain
	the motivating example highlighted in \eqref{motivatingpiroot}. 
\end{example}

\begin{example}
	Setting $(a, b, n) = \left( -\frac{5}{6}, -\frac{1}{6}, \frac{1}{2} \right)$, we may obtain that 
	\begin{align*}
	\frac{25515 \sqrt{3}}{8} = \sum_{j = 0}^{\infty} \left( \frac{4}{27} \right)^{j} & 
	\left[ \begin{matrix} 
	\frac{7}{12}, \frac{2}{3}, \frac{11}{12}, \frac{13}{12}, \frac{4}{3}, \frac{17}{12} \vspace{1mm} \\ 
	\frac{1}{2}, 1, \frac{3}{2}, \frac{3}{2}, \frac{11}{6}, \frac{13}{6} 
	\end{matrix} \right]_{j} \cdot \\
	& (14904 j^4+40824 j^3+38079 j^2+13932 j+1624). 
	\end{align*}
\end{example}

\begin{example}
	Setting $(a, b, n) = \left( -\frac{2}{3}, -\frac{1}{6}, \frac{1}{2} \right)$, we may obtain that 
	\begin{align*}
	3360 \sqrt[3]{2} = \sum_{j = 0}^{\infty} \left( 
	\frac{4}{27} \right)^{j} & \left[ \begin{matrix} \frac{7}{12}, \frac{2}{3}, \frac{5}{6}, \frac{13}{12}, \frac{7}{6}, \frac{4}{3} \vspace{1mm} \\ 
	\frac{1}{2}, 1, \frac{13}{9}, \frac{3}{2}, \frac{16}{9}, \frac{19}{9} \end{matrix} \right]_{j} \cdot \\
	& (14904 j^4+38772 j^3+34434 j^2+12039 j+1330). 
	\end{align*} 
\end{example}

\subsection{Series of convergence rate $\frac{27}{256}$}\label{subsection256}
We now are to set 
\begin{equation}\label{Fakbknk3nk}
F(n, k) = \frac{(a)_k (b)_k}{(n)_k (3 n)_k} 
\end{equation}
and to again apply our Zeilberger-based acceleration method. In this case, the evaluation of the rational certificate corresponding to \eqref{Fakbknk3nk} 
via Zeilberger's algorithm is more unwieldy compared to the $R(n, k)$-expression in Section \ref{subsectionakbknk2nk}, as is the case with the polynomials 
$p_{1}(n)$ and $p_{2}(n)$ involved in \eqref{maindifference} for the pair $(F, G)$ corresponding to \eqref{Fakbknk3nk}. So, for the sake of brevity, 
we leave it to the reader to generate such expressions from \eqref{Fakbknk3nk} using Maple, so as to verify how we may obtain a recursion of the form 
displayed in \eqref{mainrecurrence}. Accelerating $f(n)$ according to \eqref{mainrecurrence}, 
and assigning special values for the entires in the tuple $(a, b, n)$, we may obtain the following. 

\begin{example}
	Setting $(a, b, n) = \left( \frac{1}{4}, \frac{3}{4}, 1 \right)$, we may obtain that the motivating example 
	in \eqref{motivatingnewrate}  holds. 
\end{example}

\begin{example}
	Setting $(a, b, n) = 
	\left( \frac{1}{2}, \frac{1}{2}, 1 \right)$, we may obtain that 
	\begin{align*}
	\frac{327680}{\pi } = \sum_{j = 0}^{\infty} & \left( \frac{27}{256} \right)^{j} \left[ \begin{matrix} \frac{1}{2}, \frac{1}{2}, \frac{5}{6}, \frac{5}{6}, 
	\frac{7}{6}, \frac{7}{6}, \frac{3}{2} \vspace{1mm} \\ 1, \frac{5}{3}, \frac{7}{4}, 2, 2, \frac{9}{4}, \frac{7}{3} \end{matrix} \right]_{j} \cdot \\ 
	& \text{\footnotesize $(131904 j^6+777216 j^5+1886448 j^4+2413504 j^3+1716268 j^2+643120 j+99225)$}. 
	\end{align*}
\end{example}

\begin{example}
	Setting $(a, b, n) = 
	\left( \frac{1}{6}, \frac{5}{6}, 1 \right)$, we may obtain that 
	\begin{align*}
	\frac{17414258688}{77 \pi } 
	= \sum_{j = 0}^{\infty} & \left( \frac{27}{256} \right)^{j} 
	\left[ \begin{matrix} 
	\frac{1}{6}, \frac{13}{18}, \frac{5}{6}, \frac{17}{18}, \frac{19}{18}, \frac{23}{18}, \frac{25}{18}, \frac{29}{18} \vspace{1mm} \\ 
	1, \frac{3}{2}, \frac{5}{3}, \frac{7}{4}, 2, 2, \frac{9}{4}, \frac{7}{3} 
	\end{matrix} \right]_{j} \\
	& \text{\footnotesize $(96158016 j^6+566590464 j^5+1373338800 j^4+1751461056 j^3$} \\
	& \text{\footnotesize $+1238515308 j^2+459977904 j+70018325)$}.
	\end{align*}
\end{example}

\subsection{Series of convergence rate $\frac{108}{3125}$}\label{subsection3125}
We set 
\begin{equation}\label{Fakbk2nk3nk}
F(n, k) = \frac{(a)_k (b)_k}{(2 n)_k (3 n)_k}
\end{equation}
to again obtain a two-term recursion of the form indicated in \eqref{mainrecurrence}, allowing us to apply an acceleration of the form indicated in 
\eqref{mainprod}. Again, the $R(n, k)$- and $p_{1}(n)$- and $p_{2}(n)$-expressions are unwieldy in this case, relative to 
Section \ref{subsectionakbknk2nk}, 
so we again leave it to the reader to verify with Maple how these expressions may be computed so as to obtain
a hypergeometric recurrence of the form indicated in \eqref{mainrecurrence}, 
for a pair $(F, G)$ corresponding to \eqref{Fakbk2nk3nk}. 
Specializing values for the entries of the tuple $(a, b, n)$ yields new results as in the following. 

\begin{example}
	Setting $(a, b, n) = 
	\left( \frac{1}{2}, \frac{1}{2}, \frac{1}{2} \right)$, we may obtain that 
	\begin{align*}
	\frac{14175 \pi }{8} 
	= \sum_{j = 0}^{\infty} 
	& \left( \frac{108}{3125} \right)^{j} 
	\left[ \begin{matrix} 
	\frac{1}{4}, \frac{1}{4}, \frac{1}{3}, \frac{1}{3}, \frac{2}{3}, \frac{2}{3}, \frac{3}{4}, \frac{3}{4}, 1 \vspace{1mm} \\ 
	\frac{11}{10}, \frac{7}{6}, \frac{13}{10}, \frac{3}{2}, \frac{3}{2}, \frac{3}{2}, \frac{17}{10}, \frac{11}{6}, \frac{19}{10} 
	\end{matrix} \right]_{j} \cdot \\
	& \text{\footnotesize $(1737792 j^8+7561536 j^7+14017560 j^6+14432512 j^5 + $} \\
	& \text{\footnotesize $9006976 j^4+3479240 j^3+809366 j^2+103131 j+5472)$}. 
	\end{align*}
\end{example}

\begin{example}
	Setting $(a, b, n) = \left( \frac{1}{3}, \frac{2}{3}, 1 \right)$, we may obtain that 
	\begin{align*}
	\frac{162 \sqrt{3}}{\pi } = \sum_{j = 0}^{\infty} & \left( \frac{108}{3125} \right)^{j} 
	\left[ \begin{matrix} -\frac{2}{9}, -\frac{1}{6}, -\frac{1}{9}, \frac{1}{9}, \frac{1}{6}, \frac{2}{9}, \frac{4}{9}, \frac{5}{9} \vspace{1mm} \\ 
	\frac{1}{2}, \frac{3}{5}, \frac{4}{5}, 1, 1, 1, \frac{6}{5}, \frac{7}{5} 
	\end{matrix} \right]_{j} \cdot \\
	& \text{\footnotesize $(-19794537 j^6-6952473 j^5+1102491 j^4 +$} \\
	& \text{\footnotesize $258147 j^3-17244 j^2+396 j +80)$}.
	\end{align*}
\end{example}

\begin{example}
	Setting $(a, b, n) = 
	\left( \frac{1}{4}, \frac{3}{4}, 1 \right)$, we may obtain that 
	\begin{align*}
	\frac{768 \sqrt{2}}{\pi } 
	= \sum_{j = 0}^{\infty} & \left( \frac{108}{3125} \right)^{j} 
	\left[ \begin{matrix} 
	-\frac{3}{8}, -\frac{1}{4}, -\frac{1}{8}, -\frac{1}{12}, \frac{1}{12}, \frac{1}{8}, \frac{1}{4}, \frac{3}{8}, 
	\frac{5}{12}, \frac{7}{12} \vspace{1mm} \\ 
	\frac{1}{2}, \frac{3}{5}, \frac{2}{3}, \frac{4}{5}, 1, 1, 1, \frac{6}{5}, \frac{4}{3}, \frac{7}{5} 
	\end{matrix} \right]_{j} \cdot \\
	& \text{\footnotesize $(1186332672 j^8+416677888 j^7-203526144 j^6-65040384 j^5+$} \\ 
	& \text{\footnotesize $8069888 j^4+1789952 j^3-110576 j^2+1584 j+315)$}.
	\end{align*}
\end{example}

\subsection{An acceleration based on $\frac{(a)_k (b)_k}{(n)_k (a+n)_k}$}\label{subsectionaplusn}
 By setting 
\begin{equation}\label{Faplusn}
 F(n, k) = \frac{(a)_k (b)_k}{(n)_k (a+n)_k}, 
\end{equation}
 we may obtain a new hypergeometric transform that is of a similar character relative to Theorem \ref{firstmaintheorem} and that we may apply to 
 prove the motivating examples given as \eqref{55pi18} and \eqref{4thirdscube}. 

\begin{theorem}
	Writing $f(n) = \sum_{k=0}^{\infty} F(n, k)$, with $F(n, k)$ as in \eqref{Faplusn}, the recursion $$ f(n) = \frac{b-2 b n+n (-2+ a +3 n)}{(b - 2 
		n) (1+b-2 n)}+\frac{ (a - n) (b - n) (a - b + n) }{(b - 2 n) (1+b-2 n) (a+n)} f(n+1) $$ holds true if the series corresponding to $f(n)$ and 
	$f(n+1)$ converge. 
\end{theorem}

\begin{proof}
	We again set $F(n, k)$ to be as in \eqref{Faplusn}, and we proceed to apply Zeilberger's algorithm. This gives us a companion function $G(n, k) = R(n, k) 
	F(n, k)$, with $R(n, k) = -((a+n) (n (a+2 k+3 n-2)-b (k+2 n-1)))$. We find that $(F, G)$ satisfies \eqref{maindifference}, 
	with $p_{1}(n) = -a^2 b+a^2 n+a 
	b^2-a b n-b^2 n+2 b n^2-n^3$ and $p_{2}(n) = a b^2-4 a b n+a b+4 a n^2-2 a n+b^2 n-4 b n^2+b n+4 n^3-2 n^2$, according to the notation 
	in \eqref{maindifference} for the $r = 1$ case. Again, we may obtain, from \eqref{maindifference}, a relation of the form indicated in 
	\eqref{telescopemain}, with the desired result being equivalent to 
	\eqref{generalf}, again for the pair corresponding to \eqref{Faplusn}, 
	since it is easily seen that the limit shown in \eqref{requiredvanishing} vanishes. 
\end{proof}

\begin{example}
	Setting $(a, b, n) = \left( \frac{1}{4}, \frac{1}{4}, 1 \right)$, we may obtain the following formula given as Example 19 in \cite{Chu2021}: 
	$$ \frac{21 \pi }{8 \sqrt{2}} = \sum_{j = 0}^{\infty} \left( \frac{1}{4} \right)^{j} \left[ \begin{matrix} \frac{3}{4}, \frac{3}{4}, 1 \vspace{1mm} \\ 
	\frac{5}{4}, \frac{11}{8}, \frac{15}{8} \end{matrix} \right]_{j} (12 j^2+15 j+4). $$ 
\end{example}

\begin{example}
	Setting $(a, b, n) = \left( \frac{1}{3}, \frac{1}{3}, 1 \right)$, we may obtain a new proof of Example 78 from Chu and Zhang's article 
	\cite{ChuZhang2014}: $$ \frac{20 \pi }{9 \sqrt{3}} = \sum_{j = 0}^{\infty} \left( \frac{1}{4} \right)^{j} \left[ \begin{matrix} \frac{2}{3}, 
	\frac{2}{3}, 1 \vspace{1mm} \\ \frac{4}{3}, \frac{4}{3}, \frac{11}{6} \end{matrix} \right]_{j} (9 j^2+11 j+3). $$ 
\end{example}

\begin{example}
	Setting $(a, b, n) = \left( \frac{1}{6}, \frac{1}{6}, 1 \right)$, we may obtain the motivating example highlighted in \eqref{55pi18}. 
\end{example}

\begin{example}
	Setting $(a, b, n) = \left( \frac{1}{6}, \frac{1}{3}, \frac{5}{6} \right)$, we may obtain the motivating example in \eqref{4thirdscube}. 
\end{example}

\subsection{An acceleration based on $\frac{(a)_k (b)_k}{(n)_k (a+2 n)_k}$}\label{subsectionnka2nk}
We set 
\begin{equation}\label{Fnka2nk}
F(n, k) = \frac{(a)_k (b)_k}{(n)_k (a+2 n)_k} 
\end{equation}
and again apply our Zeilberger-based acceleration method. 

\begin{theorem}\label{theoremnka2nk}
	Letting $f(n) = \sum_{k=0}^{\infty} F(n, k)$, for $F(n, k)$ as in \eqref{Fnka2nk}, the recursion $f(n) = g_{1}(n) + g_{2}(n) f(n+1)$ holds true, where 
	$$ g_{2}(n) = \frac{2 (2 n+1) (a-n) (n-b) (a-b+2 n) (a-b+2 n+1)}{(a+2 n) (a+2 n+1) (b-3 n-1) (b-3 n) (b-3 n+1)}, $$ and where $g_{1}(n)$ is given by 
	the following Mathematica output. 
	\begin{verbatim}
	(a (-1 + b) b - (2 a^2 - 2 (-1 + b) b + a (-3 + b (3 + b))) n - 
	2 (-3 + a + 2 a^2 - 8 a b + b (3 + 4 b)) n^2 + (4 - 31 a + 
	38 b) n^3 - 
	46 n^4)/((-1 + b - 3 n) (b - 3 n) (1 + b - 3 n) (a + 2 n))
	\end{verbatim}
\end{theorem}

\begin{proof}
	We set $F(n, k)$ as in \eqref{Fnka2nk}.
	Using Zeilberger's algorithm, we obtain the companion function $G(n, k) = F(n, k) R(n, k)$, 
	for the rational function $R(n, k)$ given by the following Mathematica output. 
	\begin{verbatim}
	(1/(a + k + 
	2 n))(a + 2 n) (1 + a + 
	2 n) ((-1 + b) b (-1 + k) (a + k) + (2 a^2 + 3 (-1 + k) k + 
	b^2 (-2 + 5 k) + a (-3 + b (3 + b - 8 k) + 5 k) + 
	b (2 + k - 6 k^2)) n + (-6 + 4 a^2 + 8 b^2 + b (6 - 28 k) + 
	k (4 + 9 k) + a (2 - 16 b + 15 k)) n^2 + (-4 + 31 a - 38 b + 
	39 k) n^3 + 46 n^4)
	\end{verbatim}
	This gives us a difference equation of the form 
	$p_{1}(n) F(n+1,k) + p_{2}(n) F(n, k) = G(n, k+1) - G(n, k)$. 
	the polynomial $p_{1}(n)$ is given by the following Mathematica output. 
	\begin{verbatim}
	(4*a^3*b*n - 4*a^3*n^2 - 8*a^2*b^2*n + 20*a^2*b*n^2 - 12*a^2*n^3 + 
	4*a*b^3*n - 12*a*b^2*n^2 + 8*a*b*n^3 - 4*b^3*n^2 + 20*b^2*n^3 - 
	32*b*n^4 + 16*n^5 + 2*a^3*b - 2*a^3*n - 4*a^2*b^2 + 14*a^2*b*n - 
	10*a^2*n^2 + 2*a*b^3 - 10*a*b^2*n + 12*a*b*n^2 - 4*a*n^3 - 2*b^3*n + 
	14*b^2*n^2 - 28*b*n^3 + 16*n^4 + 2*a^2*b - 2*a^2*n - 2*a*b^2 + 
	4*a*b*n - 2*a*n^2 + 2*b^2*n - 6*b*n^2 + 4*n^3)
	\end{verbatim}
	The polynomial $p_{2}(n)$ is given by the following Mathematica output. 
	\begin{verbatim}
	-4*a*b*n - a^2*b + -27*a^2*n^3 - 108*n^5 + 27*a^2*b*n^2 - 
	36*a*b^2*n^2 + 108*a*b*n^3 - 4*b*n^2 - 18*b^2*n^2 + 54*b*n^3 - a*b + 
	2*b^3*n + 4*a*b^3*n - 9*a^2*b^2*n + 12*a*n^2 + 3*a*n + a*b^3 - 
	54*n^4 + 12*n^3 + 6*n^2 - 2*n*b + 27*a*b*n^2 - 9*a*b^2*n + 3*a^2*n + 
	a^2*b^3 + 4*b^3*n^2 - 36*b^2*n^3 + 108*b*n^4 - 108*a*n^4 - 27*a*n^3
	\end{verbatim}
	Summing both sides of the given difference equation with respect to $k$, 
	we may mimic our previous proofs to obtain the desired result, with \eqref{requiredvanishing} 
	again vanishing. 
\end{proof}

\begin{example}
	Setting $(a, b, n) = \left( \frac{1}{4}, \frac{1}{4}, 1 \right)$, we may obtain that 
	\begin{align*}
	\frac{155925 \pi }{128 \sqrt{2}} = \sum_{j = 0}^{\infty} \left( 
	\frac{4}{27} \right)^{j} & \left[ \begin{matrix} \frac{3}{4}, \frac{3}{4}, 1, \frac{3}{2}, \frac{3}{2} \vspace{1mm} \\ \frac{19}{12}, \frac{13}{8}, 
	\frac{23}{12}, \frac{17}{8}, \frac{9}{4} \end{matrix} \right]_{j} \cdot \\
	& (2944 j^4+11408 j^3+16288 j^2+10125 j+2304). 
	\end{align*}
\end{example}

\begin{example}
	Setting $(a, b, n) = \left( \frac{1}{3}, \frac{1}{3}, 1 \right)$, we may obtain that $$ \frac{24640 \pi }{81 \sqrt{3}} = \sum_{j = 0}^{\infty} \left( 
	\frac{4}{27} \right)^{j} \left[ \begin{matrix} \frac{2}{3}, 1, \frac{3}{2}, \frac{3}{2} \vspace{1mm} \\ \frac{14}{9}, \frac{17}{9}, \frac{13}{6}, 
	\frac{20}{9} \end{matrix} \right]_{j} (414 j^3+1323 j^2+1397 j+486). $$ 
\end{example}

\begin{example}
	Setting $(a, b, n) = \left( \frac{1}{6}, \frac{1}{6}, 1 \right)$, we may obtain that 
	\begin{align*}
	\frac{1956955 \pi }{648} = \sum_{j = 0}^{\infty} \left( 
	\frac{4}{27} \right)^{j} & \left[ \begin{matrix} \frac{5}{6}, \frac{5}{6}, 1, \frac{3}{2}, \frac{3}{2} \vspace{1mm} \\ \frac{19}{12}, \frac{29}{18}, 
	\frac{35}{18}, \frac{25}{12}, \frac{41}{18} \end{matrix} \right]_{j} \cdot \\
	& (9936 j^4+38628 j^3+55236 j^2+34315 j+7776). 
	\end{align*} 
\end{example}

\subsection{An acceleration based on $\frac{(a)_k (b)_k}{(2 n)_k (a+n)_k}$}\label{subsectionakbk2nk}
We set 
\begin{equation}\label{Fakbk2nk}
F(n, k) = \frac{(a)_k (b)_k}{(2 n)_k (a+n)_k} 
\end{equation}
and again apply our Zeilberger-based acceleration method. This may be used to obtain a hypergeometric transform much like Theorem 
\ref{theoremnka2nk}. For the sake of brevity, we leave it to the reader to fill in the details in regard to our transform obtained, via Zeilberger's 
algorithm, from \eqref{Fakbk2nk}. 

\begin{example}
	Setting $(a, b, n) = \left( \frac{1}{4}, \frac{1}{4}, \frac{1}{2} \right)$, we may obtain the motivating example highlighted in \eqref{45sqrt2}. 
\end{example}

\begin{example}
	Setting $(a, b, n) = \left( \frac{1}{3}, \frac{1}{3}, \frac{2}{3} \right)$, we may obtain that $$ \frac{320 \pi }{9 \sqrt{3}} = \sum_{j = 0}^{\infty} 
	\left( \frac{4}{27} \right)^{j} \left[ \begin{matrix} \frac{1}{2}, \frac{1}{2}, \frac{2}{3}, 1 \vspace{1mm} \\ \frac{7}{6}, \frac{11}{9}, 
	\frac{14}{9}, \frac{17}{9} \end{matrix} \right]_{j} (414 j^3+702 j^2+359 j+54). $$ 
\end{example}

\begin{example}
	Setting $(a, b, n) = \left( \frac{1}{3}, \frac{1}{6}, \frac{1}{2} \right)$, we may obtain that $$ 28 \sqrt[3]{2} = \sum_{j = 0}^{\infty} \left( 
	\frac{4}{27} \right)^{j} \left[ \begin{matrix} \frac{1}{3}, \frac{5}{12}, \frac{2}{3}, \frac{11}{12} \vspace{1mm} \\ 1, \frac{10}{9}, \frac{13}{9}, 
	\frac{16}{9} \end{matrix} \right]_{j} (621 j^3+828 j^2+303 j+26). $$ 
\end{example}

\subsection{Series of convergence rate $-\frac{1}{4}$}\label{subsectionminusquarter}
Setting 
\begin{equation}\label{Ffornegativequarter}
F(n, k) = \frac{(-1)^k (a)_k (b)_k}{(a+n+1)_k (b+n+1)_k}, 
\end{equation}
this leads us, via our acceleration method, to the following hypergeometric transform that we are to apply to obtain series of convergence 
rate $-\frac{1}{4}$. 

\begin{theorem}\label{negativequarterrec}
	Writing $f(n) = \sum_{k=0}^{\infty} F(n, k)$ for $F(n, k)$ as in \eqref{Ffornegativequarter}, the recursion 
	\begin{align*}
	& f(n) = \frac{a (3+2 b+3 n)+(1+n) (6+3 b+5 n)}{4 (1+a+n) (1+b+n)} + \\
	& \hspace{1.35cm} \frac{ (-2+a-b-n) (1+n) (2+n) (2+a-b+n) }{4 (1+a+n) (2+a+n) (1+b+n) (2+b+n)} f(n+2)
	\end{align*}
	holds true if the series corresponding to the above $f$-expressions converge. 
\end{theorem}

\begin{proof}
	We again set $F(n, k)$ to be as in \eqref{Ffornegativequarter}. Let us write $G(n, k) = F(n, k) R(n, k)$, where
	$R(n, k)$ is equal to the following rational function: 
	\begin{align*}
	& \frac{(a+n+1) (a+n+2) (b+n+1) (b+n+2)}{(a+k+n+1) (b+k+n+1)} \cdot \\
	& \left( a (2 b+2 k+3 n+3)+b (2 k+3 n+3)+6 k n+2 k (k+3)+5 n^2+11 n+6 \right).
	\end{align*}
	In this case, we find that the $r = 2$ case of \eqref{maindifference} holds, with $p_{1}(n) = (1 + n) (2 + n) ((a - b)^2 - (2 + n)^2)$ and $p_{2}(n) = -4 
	(1 + a + n) (2 + a + n) (1 + b + n) (2 + b + n)$. From the difference equation $ p_{1}(n) F(n + 2, k) + p_{2}(n) F(n, k) = G(n, k+1) - G(n, k)$, we 
	may obtain an equivalent formulation of the desired result by summing both sides with respect to $k$ and by making use of a telescoping effect, again 
	with \eqref{requiredvanishing} vanishing. 
\end{proof}

\begin{example}
	Setting $ a = \frac{1}{2}$ and $b = \frac{1}{2}$ and $n = 0$, we obtain Lupa{\c{s}}' series evaluation for $G$ \cite{Lupas2000} given in an equivalent 
	formulation below: $$ 18 G = \sum_{j = 
		0}^{\infty} \left( -\frac{1}{4} \right)^{j} \left[ \begin{matrix} \frac{1}{2}, 1, 1, 1 \vspace{1mm} \\ 
	\frac{5}{4}, \frac{5}{4}, \frac{7}{4}, \frac{7}{4} \end{matrix} \right]_{j} \left(40 j^2+56 j+19\right). $$ 
\end{example}

As indicated in the Wolfram Mathworld entry on Catalan's constant, the above formula by Lupa{\c{s}} \cite{Lupas2000} is used to calculate $G$ in the 
Wolfram Language. This motivates our generalization of Lupa{\c{s}}' formula given by Theorem \ref{negativequarterrec}, and our 
applications of this Theorem, as below. 

\begin{example}
	Setting $ a = -\frac{1}{4}$ and $b = \frac{1}{2}$ and $ n = \frac{1}{4}$, we obtain 
	$$ 21 \sqrt{2} = 
	\sum_{j=0}^{\infty} \left( -\frac{1}{4} \right)^{j} \left[ \begin{matrix} 
	\frac{5}{8}, \frac{3}{4}, \frac{9}{8} \vspace{1mm} \\ 
	1, \frac{11}{8}, \frac{15}{8}
	\end{matrix} \right]_{j} \left(80 j^2+114 j+39\right). $$
\end{example}

\begin{example}
	Setting $ a = \frac{1}{3}$ and $b = \frac{1}{2}$ and $ n = -\frac{1}{2}$, we obtain the motivating example in \eqref{5sqrt3}. 
\end{example}

\begin{example}
	Setting $ a = -\frac{5}{6}$ and $b = \frac{2}{3}$ and $ n = \frac{5}{6}$, 
	we obtain the motivating example highlighted in \eqref{motivatingcube1}. 
\end{example}

\subsection{Series of convergence rate $-\frac{1}{27}$}\label{subnegative27}
Setting 
\begin{equation}\label{Fcubed}
F(n, k) = \frac{ (a)_k^3}{ (n)_k^3} 
\end{equation}
and $ f(n) = \sum_{k=0}^{\infty} F(n, k)$, 
we may apply the procedure indicated in Section \ref{sectionmethod} 
to obtain a recurrence of the form 
\begin{equation}\label{notMarkov}
f(n) = g_{1}(n) + g_{2}(n) f(n+2), 
\end{equation}
i.e., with $r = 2$ according to the notation in \eqref{generalf}. This leads us to an acceleration identity of the form $$ f(n) = \sum _{j=0}^\infty 
\left(\prod _{i=0}^{j-1} g_{2}(n+2 i)\right) g_{1}(n + 2 j). $$ The $R(n, k)$-certificate obtained from \eqref{Fcubed} via Zeilberger's algorithm is more 
unwieldy compared to Section \ref{subsectionminusquarter}, and similarly for the required polynomial coefficients, so we leave it to the reader to verify, 
using Maple, the acceleration formula in \eqref{notMarkov}. 

\begin{example}
	Setting $a = 1$ and $ n = 2$, we may obtain an equivalent version of a series evaluation found by Amdeberhan \cite{Amdeberhan1996} in 1996: $$ 24 
	\zeta (3) = \sum_{j = 0}^{\infty} \left( -\frac{1}{27} \right)^{j} 
	\left[ \begin{matrix} 
	1, 1, 1 \vspace{1mm} \\ 
	\frac{4}{3}, \frac{5}{3}, 2 
	\end{matrix} \right]_{j} 
	\frac{56 j^2+80 j+29}{(2 j+1)^2}. $$
\end{example}

\begin{example}
	Setting $a = \frac{1}{2}$ and $n = \frac{3}{2}$, we may obtain the motivating example given in
	\eqref{motivatingzeta1}. 
\end{example}

\section{Conclusion}
We conclude by briefly considering some further subjects for exploration. 

Although our article is mainly based on the acceleration method described in Section \ref{sectionmethod}, we have discovered how recurrences as in 
\eqref{notMarkov} may be applied in conjunction with recursions that are not of the form described in Section \ref{sectionmethod} to obtain new results 
and new proofs of past results. For example, the recurrence 
\begin{align*}
& {}_{4}F_{3}\!\!\left[ 
\begin{matrix} 
x, x, x, 1 \vspace{1mm}\\ 
x + y, x + y, x + y 
\end{matrix} \ \Bigg| \ 1 \right] = \frac{p(x, y)}{3(3y-1)(3y+1) (x +y )^3} - \\ 
& \frac{y^3 (y+1)^5}{3(3y-1)(3y+1)(x+y)^3(x+y+1)^3} \, 
{}_{4}F_{3}\!\!\left[ 
\begin{matrix} 
x, x, x, 1 \vspace{1mm}\\ 
x + y + 2, x + y + 2, x + y + 2 
\end{matrix} \ \Bigg| \ 1 \right] 
\end{align*}
for the polynomial $p(x, y) = 28 y^5+2 (39 x +1) y^4+(87 x^2-3 x -2) y^3+9 x (5 x^2-1) y^2+3 
x^2 (3 x^2+2 x -3) y +3 x^3 (x -1)$ 
can be shown to be equivalent to \eqref{notMarkov}, 
and if we apply this recurrence together with 
\begin{align*}
& {}_{4}F_{3}\!\!\left[ 
\begin{matrix} 
x, x, x, 1 \vspace{1mm}\\ 
x + y, x + y, x + y
\end{matrix} \ \Bigg| \ 1 \right] = 
1 + \frac{x^3}{(x + 
	y)^3} \, 
{}_{4}F_{3}\!\!\left[ 
\begin{matrix} 
x + 1, x + 1, x + 1, 1 \vspace{1mm}\\ 
x + 1 + y, x + 1 + y, x + 1 + y
\end{matrix} \ \Bigg| \ 1 \right] 
\end{align*}
i.e., by applying the above recurrences in a successive fashion, this can be used to prove that $$ \zeta(3)=\frac{1}{16} 
\sum_{n=1}^{\infty} (-1)^{n-1} \text{\footnotesize $\frac{40885n^5 - 80079n^4 + 61626n^3 - 23366n^2 + 4374n - 324}{n^5(2n - 1)^3}$}\frac{1}{\binom{3n}{2n}^4\binom{2n}{n}}.$$ This series was found in 2005 by Mohammed \cite{Mohammed2005}. \\[0.3cm] 
Writing $s(x, y)$ in place of the ${}_{4}F_{3}$-expressions on the left-hand sides of the above recurrences, 
by writing $s(x,y)$ as a function of $s(x+2,y+2)$, this leads to the equality of $-8 \zeta(3)$ and
$$ \sum_{n=1}^{\infty} (-1)^{n} 
\text{\footnotesize $\frac{126392n^5 - 219252n^4 + 144666n^3 - 46039n^2 + 7128n - 432}{n^5(2n - 1)^3}$} \frac{1}{\binom{4n}{2n}^3\binom{3n}{n}\binom{2n}{n}}.$$ 
This series was used by Sebastian Wedeniwski \cite{Mohammed2005} to calculate Ap\'ery's constant with several million correct decimal places. 
We leave it to a separate project to fully 
explore the application of our method in Section \ref{sectionmethod} 
in conjunction with recurrences that are not of the form indicated in Section \ref{sectionmethod}. 

Our applications of \eqref{maindifference} and generalizations of this difference equation lead us to consider variants of the recurrences we have derived. 
Is there a way to use  the   WZ   method 
    to find such non-homogeneous first order linear recurrence relations? An area of a similar nature to explore has to do with the 
use of symmetry in the application of hypergeometric recursions for series accelerations. For example, we have experimentally discovered that 
\begin{align*}
& {}_{3}F_{2}\!\!\left[ \begin{matrix} x, y, 1 \vspace{1mm}\\ x+y+ \frac{1}{2}, x + y + \frac{1}{2} \end{matrix} \ \Bigg| \ 1 \right] = 
\frac{x + 2y}{x + y} \ - \\ 
& \frac{y(2y+1)^2}{(2x+2y+1)^2 (x + y)} 
{}_{3}F_{2}\!\!\left[ 
\begin{matrix} 
x, y+1, 1 \vspace{1mm}\\ 
x + y + \frac{3}{2}, x + y + \frac{3}{2} 
\end{matrix} \ \Bigg| \ 1 \right], 
\end{align*}
and we may obtain similar results via our acceleration method, and we are led to consider the above recurrence together with the corresponding recurrence 
obtained, via symmetry, by permuting $x$ and $y$. Using both recurrences, one after the other, leads us to new results as in the following: 
\begin{align*}
& \frac{\sqrt{2}}{2} = \sum_{n=0}^\infty \frac{3 + 184n -336n^2}{(4n-1)(4n-3)}\frac{\binom{4n}{2n}}{2^{10n}}, \\ 
& \frac{\pi^2}{2} = \sum_{n=1}^\infty \frac{(14n-3)(3n-1)}{n^3(2n-1)}\frac{16^n}{\binom{4n}{2n}^2\binom{2n}{n}}. 
\end{align*}


\begin{thebibliography}{99}

\bibitem{Amdeberhan1996}
 T.\ Amdeberhan, 
 \emph{Faster and faster convergent series for {{\(\zeta(3)\)}}}, 
 Electron.\ J.\ Comb.\ 3:1 (1996), \#R13. 
	
\bibitem{ApagoduZeilberger2006}
 M.\ Apagodu and D.\ Zeilberger, 
 \emph{Multi-variable {Zeilberger} and {Almkvist}-{Zeilberger} algorithms and the sharpening of {Wilf}-{Zeilberger} theory}, 
 Adv. Appl. Math.\ 37:2 (2006), 139--152. 
	
\bibitem{Bailey1935}
 W.\ N.\ Bailey, 
 \emph{Generalized Hypergeometric Series}, 
 Cambridge University Press, Cambridge, 1935. 
	
\bibitem{BaruahBerndtChan2009}
 N.\ D.\ Baruah, B.\ C.\ Berndt and H.\ H.\ Chan, 
 \emph{Ramanujan's series for $1/\pi$: a survey}, 
 Amer.\ Math.\ Monthly 116:7 (2009), 567--587 
	
\bibitem{CampbellMaple}
 J. M.\ Campbell, 
 \emph{Nested radicals obtained via the Wilf--Zeilberger method and related results}, 
 Maple Trans.\ 3:3 (2023), Article 16011 

\bibitem{CampbellLevrie2023}
 J.\ M.\ Campbell and P.\ Levrie, 
 \emph{Further WZ-based methods for proving and generalizing Ramanujan's series}, 
 J.\ Difference Equ.\ Appl.\ 29:3 (2023), pp. 366--376. 

\bibitem{ChenHouMu2012}
 W. Y. C.\ Chen, Q.-H.\ Hou and Y.-P.\ Mu, 
 \emph{The extended {Zeilberger} algorithm with parameters}, 
 J.\ Symb.\ Comput.\ 47:6 (2012), 643--654. 
	
\bibitem{Chu2011}
 W.\ Chu, 
 \emph{Dougall's bilateral {{\(_{2}H_{2}\)}}-series and Ramanujan-like {{\(\pi\)}}-formulae}, 
 Math.\ Comput.\ 80:276 (2011), 2223--2251. 
	
\bibitem{Chu2021}
 W.\ Chu, 
 \emph{Infinite series identities derived from the very well-poised {{\(\Omega\)}}-sum}, 
 Ramanujan J.\ 55:1 (2021), 239--270. 
	
\bibitem{Chu2021Ramanujan}
 W.\ Chu, 
 \emph{Ramanujan-like formulae for {{\(\pi\)}} and {{\(1/\pi\)}} via {Gould}-{Hsu} inverse series relations}, 
 Ramanujan J.\ 56:3 (2021), 1007--1027. 
	
\bibitem{ChuZhang2014}
 W.\ Chu and W.\ Zhang, 
 \emph{Accelerating {Dougall}'s {{\(_{5}F_{4}\)}}-sum and infinite series involving {{\(\pi\)}}}, 
 Math. Comput.\ 83:285 (2014), 475--512. 
	
\bibitem{ChudnovskyChudnovsky1988}
 D.\ V.\ Chudnovsky and G.\ V.\ Chudnovsky, 
 \emph{Approximations and complex multiplication according to Ramanujan}, 
 in ``Ramanujan Revisited'', Proceedings of the Centenary Conference 
 (Urbana--Champaign, 1987), eds. G. E. Andrews, B. C. Berndt, and R. A. Rankin;
 Academic Press: Boston (1988), 375--472. 
	
\bibitem{EkhadZeilberger1994}
 S.\ B.\ Ekhad and D.\ Zeilberger, 
 \emph{A {WZ} proof of {Ramanujan}'s formula for {{\(\pi\)}}}, 
 in ``Geometry, analysis and mechanics. Dedicated to Archimedes on his 2281st birthday'', 
 World Scientific: Singapore (1994), 107--108. 
	
\bibitem{Girstmair2021}
 K.\ Girstmair, 
 \emph{Reducing radicals in the spirit of {Euclid}}, 
 J.\ Symb.\ Comput.\ 104 (2021), 356--365. 

\bibitem{Guillera2006}
 J.\ Guillera, 
 \emph{Generators of some {Ramanujan} formulas}, 
 Ramanujan J.\ 11:1 (2006), 41--48. 

\bibitem{Guillera2008}
 J.\ Guillera, 
 \emph{Hypergeometric identities for 10 extended {Ramanujan}-type series}, 
 Ramanujan J.\ 15:2 (2008), 219--234. 
	
\bibitem{HessamiPilehroodHessamiPilehrood2008}
 K.\ Hessami Pilehrood and T.\ Hessami Pilehrood, 
 \emph{Simultaneous generation for zeta values by the {Markov}-{WZ} method}, 
 Discrete Math.\ Theor.\ Comput.\ Sci.\ 10:3 (2008), 115--123. 
	
\bibitem{JoyceZucker2002}
 G.\ S.\ Joyce and I.\ J.\ Zucker, 
 Special values of the hypergeometric series. {III}, 
 Math.\ Proc.\ Camb.\ Philos.\ Soc.\ 133:2 (2002), 213--222. 
	
\bibitem{KondratievaSadov2005}
 M.\ Kondratieva and S.\ Sadov, 
 \emph{Markov's transformation of series and the {WZ} method}, 
 Adv.\ Appl.\ Math.\ 34:2 (2005), 393--407. 
	
\bibitem{Koornwinder1993}
 T. H.\ Koornwinder, 
 \emph{On {Zeilberger}'s algorithm and its {{\(q\)}}-analogue}, 
 J.\ Comput.\ Appl.\ Math.\ 48:1-2 (1993), 91--111. 
	
\bibitem{LamiriWeslati2021}
 I.\ Lamiri and J.\ Weslati, 
 \emph{Limit relations involving 2-orthogonal polynomials}, 
 Integral Transforms Spec.\  Funct.\ 32:5-8 (2021), 545--559. 
	
\bibitem{LevrieCampbell2023}
 P.\ Levrie and J.\ Campbell, 
 \emph{Series acceleration formulas obtained from experimentally discovered hypergeometric recursions}, 
  Discrete Math.\    Theor.\    Comput.\     Sci.\ 19:1 (2023), dmtcs.9557. 
	
	\bibitem{LevrieNimbran2018}
	P.\ Levrie and A. S.\ Nimbran, 
	\emph{From {Wallis} and {Forsyth} to {Ramanujan}}, 
	Ramanujan J.\ 47:3 (2018), 533--545. 
	
	\bibitem{Lupas2000}
	A.\ Lupa{\c{s}}, 
	\emph{Formulae for some classical constants}, 
	in ``RoGer 2000--Bra\c{s}ov. Proceedings of the 4th Romanian--German seminar on approximation theory and its applications, 
	Bra\c{s}ov, Romania, July 3--5, 2000'', 
	Duisburg: Gerhard-Mercator-Universit{\"a}t Duisburg, Fachbereich Mathematik (2000), 70--76. 
	
	\bibitem{MansourShattuck2013}
	T.\ Mansour and M.\ Shattuck, 
	\emph{A combinatorial approach to a general two-term recurrence}, 
	Discrete Appl. Math.\ 161:13-14 (2013), 2084--2094. 
	
	\bibitem{Mingarelli2013}
	A. B.\ Mingarelli, 
	\emph{On a discrete version of a theorem of {Clausen} and its applications}, 
	Acta Math. Acad. Paedagog. Nyh{\'a}zi. (N.S.) 29 (2013), 19--42. 
	
	\bibitem{Mohammed2005}
	M.\ Mohammed, 
	\emph{Infinite families of accelerated series for some classical constants by the {Markov}-{WZ} {Method}}, 
	Discrete Math. Theor. Comput. Sci.\ 7:1 (2005), 11--24. 
	
	\bibitem{PauleSchorn1995}
	P.\ Paule and M.\ Schorn, 
	\emph{A} {Mathematica} \emph{version of {Zeilberger}'s algorithm for proving binomial coefficient identities}, 
	J. Symb. Comput.\ 20:5-6 (1995), 673--698. 
	
	\bibitem{PetkovsekWilfZeilberger1996}
	M. Petkov\v{s}ek, H. S. Wilf and D.\ Zeilberger, 
	{$A=B$}, 
	A K Peters, Ltd., Wellesley, MA, 1996. 
	
	\bibitem{Rainville1960}
	E. D. Rainville, 
	\emph{Special Functions}, 
	The Macmillan Co., Bronx, New York, 1960. 
	
	\bibitem{Ramanujan1914}
	S. Ramanujan, 
	\emph{Modular equations and approximations of $\pi$}, Quart. J. Pure Appl. Math.
	45 (1914), 350--372.
	
	\bibitem{SziklaiWeiner2023}
	P. Sziklai and Z.\ Weiner, 
	\emph{Covering all but the low weight vertices of the unit cube}, 
	J. Comb. Theory, Ser. A 193:6 (2023), Id/No 105671. 
	
	\bibitem{Wilf1999}
	H. S. Wilf, 
	\emph{Accelerated series for universal constants, by the {WZ} method}, 
	Discrete Math. Theor. Comput. Sci.\ 3:4 (1999), 189--192. 
	
	\bibitem{ZuckerJoyce2001}
	I. J. Zucker and G. S. Joyce, 
	\emph{Special values of the hypergeometric series. {II}}, 
	Math. Proc. Camb. Philos. Soc.\ 131:2 (2001), 309--319. 
	
\end{thebibliography}
\end{document}